\documentclass[10pt]{article}

\usepackage{amsmath}
\usepackage{amsfonts}
\usepackage{amssymb}
\usepackage{framed}
\usepackage{mathtools}
\usepackage{enumerate}
\usepackage{mathrsfs}
\usepackage[hidelinks]{hyperref}
\usepackage{enumitem}
\usepackage{mathrsfs}
\usepackage{scrextend}
\usepackage{amsthm}
\usepackage{bbold}
\usepackage{bbm}
\usepackage{comment}

\usepackage[T1]{fontenc}
\usepackage{lmodern}

\usepackage{thmtools}
\usepackage{thm-restate}
\usepackage{mathabx}
\usepackage[margin=1.5in]{geometry}
\usepackage{rotating}

\usepackage[nameinlink]{cleveref}

\usepackage[utf8]{inputenc}
\usepackage[T1]{fontenc}

\usepackage{xcolor}
\hypersetup{
    colorlinks,
    linkcolor={red!50!black},
    citecolor={blue!50!black},
    urlcolor={blue!80!black}
}

\usepackage{tikz-cd}

\numberwithin{equation}{section}

\theoremstyle{plain}
\newtheorem{thm}{Theorem}[section]
\newtheorem{lem}[thm]{Lemma}
\newtheorem{prop}[thm]{Proposition}
\newtheorem{cor}[thm]{Corollary}
\newtheorem{rem}[thm]{Remark}

\theoremstyle{definition}
\newtheorem{defn}[thm]{Definition}

\theoremstyle{remark}

\usepackage[nottoc,numbib]{tocbibind}

\tikzset{
  symbol/.style={
    draw=none,
    every to/.append style={
      edge node={node [sloped, allow upside down, auto=false]{$#1$}}}
  },
    labl/.style={anchor=south, rotate=90, inner sep=.5mm}
}

\newcommand\restr[2]{{
	\left.\kern-\nulldelimiterspace
	#1
	\vphantom{\big|}
	\right|_{#2}
	}}

\newcommand{\ep}{\varepsilon}

\newcommand{\ch}[1]{\widecheck{{#1}}}

\newcommand{\codim}{\operatorname{codim}}

\newcommand{\End}{\operatorname{End}}

\usepackage{tikz-cd}
\usetikzlibrary{cd}
\usepackage{leftidx}
\usetikzlibrary{positioning}
\usepackage{dsfont}
\usepackage{tikz}
\usetikzlibrary{matrix,arrows,decorations.pathmorphing,positioning}

\newcommand{\GL}{\operatorname{GL}}

\usepackage{amsmath,calligra,mathrsfs}

\usepackage[cal=boondoxo]{mathalfa}

 \newcommand{\Addresses}{{
  \bigskip
  \footnotesize

\textsc{CNRS, IMJ-PRG, Sorbonne Universit\'{e}, 4 place Jussieu, 75005 Paris, France}\par\nopagebreak
  \textit{E-mail address}, G.~Baldi: \texttt{baldi@imj-prg.fr} 

  \medskip

\textsc{IAS, Institute for Advanced Study. 1 Einstein Dr, Princeton, 08540 New Jersey (United States)}\par\nopagebreak
  \textit{E-mail address}, D.~Urbanik: \texttt{urbanik@ihes.fr}
}}

\DeclareMathOperator{\sheafhom}{\mathcal{H \kern -1pt o \kern -2pt m}}
\DeclareMathOperator{\sheafper}{\mathcal{P \kern -1pt e \kern -2pt r}}
\DeclareMathOperator{\sheafiso}{\mathcal{I \kern -1pt s \kern -2pt o}}
\DeclareMathOperator{\sheafend}{\mathcal{E \kern -1pt n \kern -2pt d}}
\DeclareMathOperator{\sheafaut}{\mathcal{A \kern -1pt u \kern -2pt t}}

\newcommand{\C}{\mathbb{C}}
\tikzset{
  trim node/.default=1cm,
  trim node/.style={
    overlay,
    append after command={
      ([xshift={+#1}]\tikzlastnode.north west)
      ([xshift={+-#1}]\tikzlastnode.south east)}},
  down and trim/.default=1cm,
  down and trim/.style={
    yshift=-(\pgfmatrixcurrentcolumn-1)*1.5\baselineskip,
    trim node={#1}},
  downup and trim/.default=1cm,
  downup and trim/.style={
    yshift=iseven(\pgfmatrixcurrentcolumn) ? -1.5\baselineskip : 0pt,
    trim node={#1}},
  -|/.style={to path={-|(\tikztotarget)\tikztonodes}},
  |-/.style={to path={|-(\tikztotarget)\tikztonodes}},
  -| sl/.style={-|, xslant=-1},
  |- sl/.style={|-, xslant= 1},
  center picture/.style={
    trim left=(current bounding box.center),
    trim right=(current bounding box.center)}}

\newcommand{\N}{\mathbb{N}}
\newcommand{\Z}{\mathbb{Z}}
\newcommand{\R}{\mathbb{R}}
\newcommand{\V}{\mathbb{V}}

\usepackage{microtype}

\title{Intersections and the Bézout Range: Abelian Varieties}\date{\today}
\author{Gregorio Baldi and David Urbanik}

\begin{document}

\maketitle

\begin{abstract}
Given subvarieties $X, Y$ of a complex algebraic variety $S$ of complementary dimension, must they intersect? When $S$ is projective space, this is a consequence of the classical Bézout theorem, and an analogue for simple abelian varieties was established by Barth in 1968. Moreover, the moving lemma suggests that, after suitable translations, one may arrange for intersections of the expected dimension.

In this work, we obtain variants for simple abelian varieties in the spirit of the completed Zilber--Pink philosophy. When $X$ and $Y$ have complementary dimension, we show that the intersections $X \cap [n]Y$ are zero-dimensional for all but finitely many integers $n$, and that these intersections collectively give rise to an analytically dense subset of $X$ as $n$ varies. We moreover control those $n$ for which $X \cap [n] Y$ has a positive dimensional component uniformly in $X, Y$ and $A$. When $\dim X + \dim Y < \dim A$, we show that $X \cap [n]Y = \varnothing$ for a set of integers $n$ of asymptotic density one, except in the presence of intersections at torsion points.
\end{abstract}

\tableofcontents

\section{Introduction}

We start by recalling some general considerations about intersection theory, cf. for example the introduction of Fulton's book \cite{zbMATH01027930}.
\begin{defn}\label{meetproperly}
Let $S$ be a smooth complex variety and $X,Y\subset S$ be irreducible Zariski closed subvarieties.
We say that $X$ and $Y$ \emph{meet properly} (or intersect properly) in $S$ if every irreducible component $Z$ of $X\cap Y$ satisfies
\begin{equation}
  \dim Z \;=\; \dim X + \dim Y - \dim S,
\end{equation}
i.e., each such component has the \emph{expected dimension}. We write $(X\cap Y)^\text{e}$ for the union of the irreducible components of the expected dimension, and $(X\cap Y)^\text{ue}$ for the others, namely the one of \emph{unexpected} dimension (i.e. of dimension $> \dim X + \dim Y - \dim S$).
\end{defn}
We stress that meeting properly is, in general, a much more stringent condition than the mere existence of a component of $X \cap Y$ of the expected dimension.

\begin{rem}\label{rmkdim} Our convention is that $\dim \emptyset = -\infty$, and one considers $X \cap Y=\emptyset$ to be expected when $\dim X + \dim Y <  \dim S$.
\end{rem}

\noindent  We think of $S$ and $X$ as fixed, and study the distribution of the expected components of $X \cap Y$ as $Y$ varies in a suitable way. This point of view is closely related to a cornerstone of intersection theory:
\begin{thm}[Chow's Moving Lemma \cite{zbMATH03122792, zbMATH01027930}]  
\label{chowmovinglem}
Let $X, Y \subset S$ be closed subvarieties of a smooth quasi-projective variety $S$.  
Then there is an algebraic cycle $Y'$ on $S$, rationally equivalent to $Y$, such that $Y'$ and $X$ meet properly.  
\end{thm}  

We remark here that one cannot in general ensure that $Y'$ is effective. In this paper we investigate related properties for the case where $S$ is an abelian variety. We plan to consider the case where $S$ is a Shimura variety in a sequel. Before stating our main results, we recall what is known about the following problem: if $X, Y \subset S$ have dimensions adding up to at least $\operatorname{dim} S$, must they intersect?

\subsection{Motivations: B\'{e}zout for projective spaces and simple abelian varieties}

In the case $S=\mathbb{P}^n_\C$ (or in fact when $S=\mathbb{P}^n_k$ with $k$ an arbitrary field) one has the following classical theorem named after B\'ezout. (The case of curves was already known by I. Newton in 1687, cf. the proof of Lemma 28 in volume 1 of his \emph{Principia}.) For a general and modern perspective we refer to \cite[Thm. 12.3 and page 223]{zbMATH01027930}. 

\begin{thm}
\label{bezoutplus}
Let $X, Y \subset \mathbb{P}^{n}_{\C}$ be irreducible subvarieties.
\begin{enumerate}
\item[(1)] If $\dim X + \dim Y \geq n$, then $X\cap Y \neq \emptyset$.
\item[(2)] For a general translate $g \cdot Y$ with $g \in \operatorname{PGL}_{n+1}(\C)$ the varieties $X$ and $g\cdot Y$ meet properly.
\end{enumerate}
\end{thm}
\noindent Moreover when $X$ and $Y$ meet properly the degree of the intersection is the product of the degrees, generalizing the fact that, if $\dim X + \dim Y=n$, the number of intersection points counted with multiplicity equals the product $\deg X \cdot \deg Y$.

The analogue of \autoref{bezoutplus}(1) is also known in the case $S=A$ a simple abelian variety thanks to the work of Barth \cite{zbMATH03256696}; simplified proofs that work for any field of characteristic zero were later given by Prasad \cite{zbMATH00495533} and Debarre \cite{zbMATH00807396}. More recently Debarre and Moonen \cite{2025arXiv250914940D} have obtained the same result over an arbitrary base field $k$ as well as well as a version, under an additional non-degeneracy assumption for $X$, in the non-simple case (cf. \S5 in \emph{op. cit.}).

\begin{thm}[Barth]\label{bezab}
Let $A$ be a simple complex abelian variety. If $\dim X+\dim Y \geq  \dim A$, then the subvarieties $X$ and $Y $ must meet: $X \cap Y \neq \emptyset$.
\end{thm}
Of course, the above conclusion cannot be strengthened to assert that $X$ and $Y$ meet properly, since, for instance, $X$ might be contained in $Y$. Nevertheless, in analogy with \Cref{bezoutplus}(2) and \Cref{chowmovinglem}, one may ask whether there exists a natural way to ``move'' $Y$ so that one obtains a proper intersection. On an abelian variety $A$, one can naturally translate by points $a \in A$, replacing $X \cap Y$ with $X \cap (Y + a)$. In this case the problem, as well as its generalization to an arbitrary algebraic group, is fully understood as a consequence of Kleiman's Theorem \cite{zbMATH03451984}. Instead, our main interest lies in acting via \emph{group endomorphisms} of $A$, whose behaviour is typically more chaotic. We also plan a subsequent paper with $S$ a Shimura variety, where instead of endomorphisms of $S$ one uses Hecke correspondences.

\subsection{Main results}

Our results come in three flavours:
\begin{itemize}
\item[-] those which apply when $\dim X + \dim Y \geq \dim A$, i.e., when we are ``inside the B\'ezout range'' (see \Cref{mainthmB}); 
\item[-] those which apply when $\dim X + \dim Y < \dim A$; i.e., when we are ``outside the B\'ezout range'' (see \Cref{densityofnonintthm}); and
\item[-] general considerations about uniformity in the setting of the so called \emph{geometric Zilber--Pink conjecture} (see \Cref{uniformityinabeliansetting} and \Cref{rhounigeoZP}).
\end{itemize}
We now describe them in detail. (We work over the complex numbers and omit the field $k$ from the notation.)
\subsubsection{Inside the B\'ezout Range}

Our first main result establishes a version of \autoref{bezoutplus}(2) for abelian varieties, describing the behaviour of the expected intersections between $X$ and the collection $\{ [n] Y \}_{n \in \N}$.
\begin{thm}\label{thmA}
Let $A$ be a simple abelian variety and $X,Y$ irreducible Zariski closed subvarieties such that $\dim X+\dim Y \geq  \dim A$. Then for all but finitely many $n \in \Z \subset \operatorname{End} (A)$, $X$ and $[n]Y$ meet properly, in the sense of \Cref{meetproperly}. 
\end{thm}
\noindent Note that the simplicity assumption is required for the statement, as can be seen by considering a product situation $A = A' \times A''$ with $\dim A' \geq \dim A''$ and taking both $X$ and $Y$ to be translates of $A' \times \{ 0 \}$.

\begin{rem}
The above is interesting also in the special case where $X=Y$. Indeed this is the setting appearing in Hindry's method \cite{zbMATH04040085} for proving Manin-Mumford (as well as the related Edixhoven-Yafaev approach to the André-Oort conjecture under GRH). In such cases one wishes to intersect a subvariety $X$ with a translate $[n]X$, observe that torsion points are in $X \cap [n]X$, and then argue by comparing the degree bound from Bézout to the size of Galois orbits. 
\end{rem}

The idea of the proof of the above statement is to relate the problem to the so-called ``completed Zilber–Pink philosophy'' for subvarieties of abelian varieties (cf. \cite{MR3552014}, the recent monographs \cite{zbMATH06022187, zbMATH07516542}, and our earlier work \cite{2024arXiv240616628B} for the most general case). This perspective allows us to reformulate the question as a problem of likely and unlikely intersections inside the product $A \times A$. One is then led to analyze both the expected and the unexpected intersections between $Y \times X$ and the \emph{special subvarieties of $A \times A$}—namely, torsion translates of subabelian varieties of $A \times A$—a class that includes, in particular, the graphs of endomorphisms of $A$ (and especially those of $[n] : A \to A$, denoted by $\Gamma_n$).

Using this reformulation we also discuss several refinements of the completed Zilber–Pink philosophy which, to our knowledge, have not been made explicit before. Namely those motivated by the following two questions:
\begin{itemize}
\item To what extent is the above finiteness uniform?
\item Are unions of the form
\begin{displaymath}
\bigcup_{i } (X \cap [n_{i}] Y)^{\operatorname{e}},
\end{displaymath}
with $(n_{i})_{i \in \N}$ some infinite sequence of endomorphisms, (analytically or Zariski) dense in $X$?
\end{itemize}

\noindent In this direction, we prove the following result. Below we denote by $\chi (\mathcal{L})=\dim H^0(A,\mathcal{L})$ the degree of the polarization.

\begin{thm}\label{mainthmB}
Let $(A,\mathcal{L})$ be a simple polarized abelian variety and $X,Y$ irreducible subvarieties such that $\dim X+\dim Y \geq  \dim A$. \begin{enumerate}
\item[(1)] There exists a constant
\begin{equation}
C=C(\deg_\mathcal{L} (X), \deg_\mathcal{L} (Y), \chi (\mathcal{L}), \dim A)
\end{equation}
such that for all but at most $C$ elements $\lambda \in  \operatorname{End} (A)$, $X$ and $\lambda Y$ meet properly.
\item[(2)] For any infinite subset $\{ n_{i} \}_{i \in \N} \subset \mathbb{Z} \subset \operatorname{End}(A)$, the set
\begin{displaymath}
\bigcup_{i \in \N  } (X \cap [n_{i}] Y)^{\operatorname{e}}
\end{displaymath}
is analytically dense in $X$.
\end{enumerate}
\end{thm}

\begin{rem}
Instead of considering the subvarieties $[\lambda] Y$ in \autoref{mainthmB}, one could also work with $[\lambda]^{-1} Y$; the same statements also hold in that case. The latter point of view is perhaps more similar to the Shimura setting where one works with true correspondences rather than functions. However since the density statements are stronger when one works with $[\lambda] Y$, and arguably more natural, we have chosen the former version.
\end{rem}
The proof of the above theorem employs a variety of techniques. The uniformity statement appearing in part (1) builds on a generalization of a theorem of Bogomolov \cite[Thm. 1, p. 58]{zbMATH03730302}. In fact, it will follow from a much more general result, \Cref{rhounigeoZP}, which makes the geometric Zilber–Pink theorem established in \cite[Thm. 7.1]{2024arXiv240616628B} uniform (see indeed \Cref{secunif} for more general results, which might be of independent interest). By contrast, the proof of part (2), which is given in \Cref{equidi}, relies on more analytic techniques inspired by equidistribution results in the tangent bundle. 

\subsubsection{Outside the B\'ezout Range}
\label{introoutsidesec}

If $\dim X + \dim Y < \dim A$, it still makes sense to ask if $X$ and $Y$ meet properly, which, as explained in \Cref{rmkdim}, happens if and only if $X \cap Y = \varnothing$. Note however that the statements appearing in \autoref{thmA} and \autoref{mainthmB} are all false in this case. For instance, if $X$ and $Y$ both contain the identity $0 \in A$, then $X$ and $[n] Y$ always intersect at $0$. More generally if $X$ and $Y$ contain torsion points $t$ and $t'$ which generate the same cyclic subgroup, then $X$ and $[n] Y$ will intersect for $n$ lying in some subset of $\N$ of positive density. 

To understand what is happening we recall the following result of Raynaud \cite{zbMATH03929174}, often referred to as the Manin-Mumford conjecture. Nowadays such a theorem has many proofs, and we refer to the one of Pila-Zannier \cite{zbMATH05292756} where various ideas related to the ones used in \Cref{densityofnonintthm} appeared for the first time. 
\begin{thm}[{Raynaud \cite{zbMATH03929174}}]
Let $A$ be a simple abelian variety and $Z \subset A$ a strict algebraic subvariety. Then $Z$ contains finitely many torsion points.
\end{thm}
It follows from the above that $Y \times X\subset A \times A$ contains finitely many pairs $\{ (t_{i}, t'_{i}) \}^{m}_{i=1}$ with $t_{i}$ (resp. $t'_{i}$) a torsion point of $A$ contained in $Y$ (resp. $X$). We observe that
\[ (t_{i}, t'_{i}) \in \Gamma_{n}\textrm{ for some }n \hspace{1em} \iff \hspace{1em} t'_{i} \in \langle t_{i} \rangle \]
where $\langle t_i \rangle$ denotes the subgroup of $A$ generated by $t_i$. Let $k_{i} := |\langle t_{i} \rangle|$ be the cardinality. Write $I \subset \{ 1, \hdots, m \}$ for those $i$ such that $t'_{i} \in \langle t_{i} \rangle$, and for each $i \in I$ write $0 \leq e_{i} < k_{i}$ for the smallest $e_{i}$ such that $e_{i} t_{i} = t'_{i}$. Then $\Gamma_{n}$ contains the pair $(t_{i}, t'_{i})$ if and only if
\begin{equation}
\label{modularcondition}
n \equiv e_{i} \pmod{k_{i}} .
\end{equation}
If one considers all equations (\ref{modularcondition}) for all $i \in I$ simultaneously, the subset of $\N$ which solves at least one of these equations has a natural density $\delta \in \mathbb{Q}_{\geq 0} \cap [0,1]$. 

Setting $V:= \bigcup _{i \in I} \{t_i'\}\subset X$, our main result outside the Bézout range is the following:
\begin{thm}
\label{densityofnonintthm}
For a density $1$ subset of $n \in \N$, the intersection $X \cap  [n] Y$  lies in $V$. In particular for a density $1 - \delta$ subset of $\N$, we have $X \cap [n] Y = \varnothing$.
\end{thm}
In the above, density is understood as \emph{natural density}; that is 
\begin{displaymath}
| \{n \in \N  : X \cap  [n] Y\subset V\}|/n \to 1, \text{  as  }  n \to + \infty. 
\end{displaymath}
The proof uses ideas from o-minimality and is inspired by the so-called Pila-Zannier strategy.

\begin{rem}
Though we decided to state \Cref{densityofnonintthm} in terms of natural density, the proof actually gives a more refined result. Namely that the cardinality of the ``bad'' $n$ for which $X \cap [n] Y$ is neither empty nor lies in $V$ is at most $O(n^\epsilon)$ for any $\epsilon >0$. In fact the recent solution of Wilkie's conjecture \cite{zbMATH07817080}, enables one to replace $n^{\ep}$ with $(\log n)^{a}$ for some fixed exponent $a > 0$.
\end{rem}

\subsubsection{Uniformity}\label{secunif}
To conclude the introduction, we briefly comment on the uniformity aspect of our ``inside the range'' theorems and record a more general result in this direction. Results of this nature have appeared previously; see, for instance, \cite{zbMATH02150540}, \cite[Thm. 1.8]{zbMATH07393134}, and \cite[Chap. 24]{zbMATH07516542}. However, the applications considered here require a level of generality not covered by the existing literature. In the setting of abelian varieties, the general result \Cref{rhounigeoZP} specializes to the following statement.

\begin{defn}
A \emph{family} (of complex algebraic varieties) is simply a complex algebraic map $f : B \to T$ of complex algebraic varieties. A \emph{subfamily} of $f$ is a pair $(\iota : C \hookrightarrow B, g : C \to T)$ of complex algebraic maps, where $\iota$ is a locally closed embedding and $g = f \circ \iota$.
\end{defn}

\begin{thm}\label{uniformityinabeliansetting}
Let $f : A \to T$ be a family of abelian varieties, and let $\rho : S \to T$ with $\iota : S \hookrightarrow A$ be a subfamily of algebraic subvarieties of the fibres of $f$. Then there exists a finite collection $(\iota_{i}, g_{i})$ of subfamilies of $f$ such that the fibres of the $g_{i}$ are exactly the \emph{maximal atypical intersections} between the fibres of $\rho$ and translates of abelian subvarieties of the fibres of $f$.
\end{thm}

The phrase ``maximal atypical intersections'' is likely new to many readers. To explain where it comes from, we give a brief overview of the Zilber-Pink philosophy in \S\ref{typvsatyp}. To relate to the setting of \autoref{thmA} and \autoref{mainthmB}, the point is that each non-proper intersection between $X$ and $[n] Y$ will give rise to a non-proper intersection -- necessarily of positive dimension -- between $Y \times X$ and the graph of $[n]$, and the latter type of intersection will be a ``maximal atypical'' intersection between $Y \times X$ and this graph in our language. One can then recover \autoref{mainthmB}(1) from \autoref{uniformityinabeliansetting} by choosing $f$ to be an appropriate family of self-product abelian varieties, letting $\rho$ be a family of varieties of the form $Y \times X$, and then analyzing the possibilities for the $g_{i}$. Finally for recent results related to the uniformity of the \emph{arithmetic part of the Zilber-Pink conjecture} we refer to the survey \cite{2021arXiv210403431G}.

\subsection{Plan of the paper}

In \Cref{sec2} we review two pieces of background. In the first subsection \S\ref{typvsatyp} we describe in broad terms the language of unlikely intersection theory and how it specializes to our situation. In the second subsection \S\ref{periodtorsdefforA} we review the construction of the Poincar\'e torsor and explain how to recover from it a local system that will be important in our proof of the uniform theorem \autoref{uniformityinabeliansetting}. In \S\ref{sec:deg} we review some standard facts about degrees of varieties.

In \Cref{sec3} we give a complete proof of \autoref{thmA} and \autoref{mainthmB} (2). This section can be read independently of \S\ref{typvsatyp} and \S\ref{periodtorsdefforA}. 

In \Cref{sec4} we prove the uniform geometric Zilber-Pink conjecture for variations of mixed Hodge structures. In addition to specializing to \autoref{uniformityinabeliansetting} above, the extra generality will allow us to also study a uniform version of the Shimura variety variant of this problem in a subsequent article; we also expect it to have other applications. In particular we make our arguments in \Cref{sec3} uniform.

In the final \Cref{finalsec}, we prove \Cref{densityofnonintthm}. This section can be read independently and relies on more arithmetic ideas, rooted in Diophantine geometry, in contrast with the more geometric approach of the preceding sections.

\subsection*{Acknowledgements}
We thank C. Daw for a related and very helpful discussion and J. Tsimerman for a discussion surrounding the ideas behind \Cref{abthmsimplecase}(2). We are grateful also to G.  Binyamini, O. Debarre, Ph. Engel, Z. Gao, S. Grushevsky, J. Pila, and A. Shankar for their interest.

This work was done while the authors were at the Institute of Advanced Study in Princeton and they would like to thank the institute for its hospitality and for providing excellent working conditions. They are also grateful for the support of the Ambrose Monell Foundation as well as the Giorgio and Elena Petronio Fellowship Fund.

G.B. was partially supported by the grant ANR-HoLoDiRibey of the Agence Nationale de la Recherche.

\section{General background}\label{sec2}
\subsection{The language of unlikely intersection problems}\label{typvsatyp}

In unlikely intersection theory, and in particular when studying conjectures of ``Zilber-Pink type'', one usually considers a quadruple $(Z, \mathcal{S}, \mathcal{W}, Y)$, where
\begin{itemize}
\item[-] the variety $Z$ is a complex variety, the \textbf{ambient variety}, typically smooth and quasi-projective;
\item[-] the set $\mathcal{S}$ is a collection of irreducible subvarieties of $Z$, with $Z \in \mathcal{S}$, the \textbf{special varieties};
\item[-] the set $\mathcal{W} \supset \mathcal{S}$ is a larger collection of irreducible subvarieties of $Z$, the \textbf{weakly special} varieties; and
\item[-] the variety $Y \subset Z$ is an irreducible closed subvariety which is not contained in any element of $\mathcal{S}$ other than $Z$ itself.
\end{itemize}
One then studies intersections between $Y$ and the special and weakly special subvarieties of the ambient variety $Z$. The dimensions of these varieties should then predict when and how often $Y$ intersects elements of $\mathcal{S}$ and $\mathcal{W}$.

\begin{defn}
Suppose $W \in \mathcal{W}$. An irreducible component $C \subset Y \cap W$ is called an \emph{atypical} intersection (with $Y$) if $\codim_{Y} C < \codim_{Z} W$. It is called a \emph{typical special} intersection if $W \in \mathcal{S}$, we have $\codim_{Y} C = \codim_{Z} T$, and $C$ does not arise as an atypical component of $Y \cap W'$ for some other $W' \in \mathcal{S}$. More generally $C$ is a \emph{typical weakly special} intersection if we have $\codim_{Y} C = \codim_{Z} W$, and $C$ does not arise as an atypical component of $Y \cap W'$ for some other $W' \in \mathcal{W}$.
\end{defn}

\begin{rem}
The above is related to the definition of expected and unexpected components appearing in \Cref{meetproperly}. However here it is used in a much more rigid setting, where one is allowed to intersect only against special or weakly special subvarieties. \Cref{beginiingabbez} is devoted to the study of the precise relationship between the two notions.
\end{rem}

\begin{defn}
Given a quadruple $(Z, \mathcal{S}, \mathcal{W}, Y)$ as above, a \emph{special} (resp. \emph{weakly special}) subvariety of $Y$ is an irreducible component of an intersection $Y \cap W$ for $W \in \mathcal{S}$ (resp. $W \in \mathcal{W}$). We also call such subvarieties special (resp. weakly special) intersections.
\end{defn}

A special (resp. weakly special) subvariety of $Y$ is called strict if it is not equal to $Y$. One is interested in the following two statements for varying $Y$:
\begin{itemize}
\item[(1)] The \textbf{Zilber-Pink statement}: the maximal-under-inclusion strict atypical special subvarieties of $Y$ form a finite set.
\item[(2)] The \textbf{geometric Zilber-Pink statement}: the maximal-under-inclusion strict atypical weakly special subvarieties of $Y$ belong to finitely many algebraic families of subvarieties of $Y$.
\item[(3)] The collection of typical special intersections with $Y$ is dense in $Y(\mathbb{C})$ if and only if there exists at least one such intersection.
\end{itemize}
The geometric Zilber-Pink statement is known in broad generality --- in particular whenever the pair $(\mathcal{W}, Y)$ comes from an admissible variation of mixed Hodge structures --- as a consequence of \cite{2024arXiv240616628B}; see also \S\ref{unifgeozpsec}. On the other hand the Zilber-Pink problem remains open even in some of the simplest cases, for example when $Z = \mathbb{G}^{r}_{m}$ and the special subvarieties are given by torsion translates of algebraic subgroups.

In the situations which will be of interest to us, both in this paper and a paper to follow, the above setup specializes as follows:

\paragraph{Abelian Varieties:} If $Z = A$ is an abelian variety, then $\mathcal{S}$ is the collection of all translates of abelian subvarieties of $A$ by torsion points (so just $A$ itself and the collection of torsion points if $A$ is simple), and $\mathcal{W}$ is the collection of all translates of abelian subvarieties of $A$ by arbitrary points (so in particular contains all points of $A$). 

\paragraph{Shimura Varieties:} If $Z = S$ is a Shimura variety, then $\mathcal{S}$ is the collection of all irreducible components of images under Shimura morphisms $S' \to S$. The set $\mathcal{W}$ is in general larger, and contains in addition all subvarieties obtained as irreducible components of maps of the form $S_{1} \times \{ s_{2} \} \to S$, where $S_{1} \times S_{2} \to S$ is a Shimura morphism. (Note that one can take $S_{1}$ to itself consist of a point, so all points of $S$ are again weakly special.) An easy example to keep in mind is $\mathcal{A}_1\times \mathcal{A}_1\to \mathcal{A}_2$, where $\mathcal{A}_{g}$ is the moduli space of principally polarized abelian varieties.

\paragraph{Weakly Specials of a Local System:} Given an irreducible algebraic variety $Z$ with a local system $\mathbb{V}$ on $Z$, there is an alternative notion of a weakly special subvariety $W \subset Z$ for $(Z, \mathbb{V})$: this is an irreducible algebraic subvariety $W$ which is maximal for the dimension of its algebraic monodromy group $\mathbf{H}_{W}$. Here $\mathbf{H}_{W}$ is the identity component of the algebraic group
\[ \overline{\textrm{im}\left[ \pi_{1}(W^{\textrm{nor}},w) \to \GL(\mathbb{V}_{w}) \right]}^{\textrm{Zar}} \]
where $w$ is a point of the normalization $W^{\textrm{nor}} \to W$ of $W$. Note that $\mathbf{H}_{W}$ is independent of $w\in W$ as an abstract group. 

In situations of interest, the set $\mathcal{W}$ referred to above can usually be described in this way for some local system $\mathbb{V}$ on $Z$; this will always be the case for us. In the context where $Z = A$ is an abelian variety, the local system $\mathbb{V}$ is the one coming from the Poincar\'e torsor, and is reviewed in \S\ref{periodtorsdefforA}.

\paragraph{References:} Another overview of the ``atypical-vs-typical'' perspective can be found for example in \cite[Sec. 2.2]{2023arXiv231211246B}. Other references include: 
\begin{itemize}
\item For results concerning atypical intersections: \cite{2021arXiv210708838B} as well as \cite{2024arXiv240616628B} for an effective proof of geometric Zilber-Pink that works in the more general mixed setting. Working in the setting of a VMHS (as opposed to working only with pure VHSs) is necessary to obtain results about abelian varieties. One of the first references handling this case explicitly is \cite{MR3552014}.

\item For results concerning typical intersections: \cite{2023arXiv230316179K, 2022arXiv221110592E} for the case of pure variations of Hodge structures. Related to this viewpoint in the mixed setting is also the study of likely intersection problems for modular curves and torsion in abelian schemes, cf. \cite{zbMATH07259017} and the later work of Gao, for instance \cite{zbMATH07305885}.
\end{itemize}

\subsection{The Poincar\'e bundle}\label{periodtorsdefforA}

Let $A$ be an abelian variety, and write $A^{\vee}$ for its dual. The variety $A^{\vee}$ is the solution to the moduli problem which parameterizes degree zero line bundles on $A$. More precisely, it is a variety such that for each complex algebraic variety $T$ we have a natural bijection
\begin{align*}
A^{\vee}(T) &:= \operatorname{Hom}(T, A^{\vee}) \\
&\simeq \left\{ \mathcal{L} \in \operatorname{Pic}(A \times T) : \deg \restr{\mathcal{L}}{A \times \{ t \}} = 0 \hspace{1em} \forall t \in T(\mathbb{C}) \hspace{0.25em} \textrm{ and } \hspace{0.25em} \restr{\mathcal{L}}{\{ 0 \} \times T} \textrm{ is trivial} \right\} .
\end{align*}
The universal line bundle $\mathcal{P} \to A \times A^{\vee}$ which corresponds to the identity morphism in $\operatorname{Hom}(A^{\vee}, A^{\vee})$ is called the Poincar\'e bundle (see also \cite[Sec. 4]{zbMATH07206871}). If we remove its zero section we obtain a $\mathbb{G}_{m}$-torsor $\mathcal{P}^{\times} \to A \times A^{\vee}$ with the property that $\rho : \mathcal{P}^{\times} \to A$ has the structure of a family of semi-abelian varieties over $A^{\vee}$ with torus rank one. Let $g : \mathcal{P}^{\times} \to A \times A^{\vee}$ be the natural map, let $T$ be the corresponding family of kernels over $A$, and write $\textrm{pr} : A \times A^{\vee} \to A$ for the natural projection.

Consider the cohomology local system $\mathbb{V} := R^{1} \rho_{*} \mathbb{Z}$. It fits into an exact sequence
\begin{equation}
\label{PTexactseq}
0 \to R^{1} \textrm{pr}_{*} \mathbb{Z} \to \mathbb{V} \to R^{1} \left(\restr{\rho}{T}\right)_{*} \mathbb{Z} \to 0 .
\end{equation}
Note that $R^{1} \textrm{pr}_{*} \mathbb{Z}$ is trivial of rank $2g$, and $R^{1} \left(\restr{g}{T}\right)_{*} \mathbb{Z}$ is trivial of rank $1$. The local system $\mathbb{V}$ of rank $2g+1$ has monodromy group isomorphic to $\mathbb{Z}^{2g}$. Indeed, the fibres of of the exact sequence (\ref{PTexactseq}) can be identified with mixed Hodge structure extensions of $H^{1}(A, \mathbb{Z})$ by $\mathbb{Z}(-1)$, which by \cite[\S3]{zbMATH03737806} (cf. \cite[Prop. 4.3.14]{filipnotes}) are naturally identified with points of $A^{\vee}$. If we dualize the sequence (\ref{PTexactseq}) we obtain a local system $\mathbb{V}^{\vee}$ whose fibres are extensions of $\mathbb{Z}(1)$ by $H_{1}(A, \mathbb{Z})$ and are naturally identified with points of $A$. Under this identification the VMHS $\mathbb{V}^{\vee}$ is the tautological one, and its monodromy group is identified with the fundamental group of $A$. 

The local system $\mathbb{V}^{\vee}$ will be useful because it will give us another way of understanding weakly specials:

\begin{prop}
\label{abelianweakspthesame}
The weakly special subvarieties of $A$ are exactly the weakly special subvarieties of $A$ for the local system $\mathbb{V}^{\vee}$.
\end{prop}

\begin{proof}
This follows from the proof of \cite[Prop. 5.1]{zbMATH05934826}.
\end{proof}

The entire construction also works in the relative setting where one considers the universal family $u : \mathbb{A}_{g} \to \mathcal{A}_{g}$ of principally polarized abelian varieties (at least after fixing a suitable level structure). In this case one obtains a local system $\mathbb{V}^{\vee}$ on $\mathbb{A}_{g}$ which specializes to the local system constructed above after restricting to a fibre of $u$.

\subsection{Degree}\label{sec:deg}

We will be interested in the situation where $S$ is a quasi-projective variety with compactification $\overline{S}$, $\mathcal{L}$ is an ample line bundle on $\overline{S}$, and we want to associate a degree to subvarieties $V$ of $S$ (cf. \Cref{unifgeozpsec}); for general reference we follow \cite[Ch.~1, pp.~15--17]{zbMATH02134816}. In this case:

\begin{defn}
\label{degreesubdef}
Given a subvariety $V \subset S$ we denote by $\deg_\mathcal{L} V$ the quantity $\deg_{\mathcal{L}} \overline{V}$, where $\overline{V}\subset \overline{S}$ is the Zariski closure of $V$.
\end{defn}

On any algebraic variety $S$ with projective compactification $(\overline{S}, \mathcal{L})$, the algebraic subvarieties of $S$ of degree at most $d$ belong to finitely many algebraic families. More precisely, there is a natural scheme $\textrm{Var}(S)$, as for instance in \cite[\S4.7]{zbMATH07673321}, which parametrizes geometrically integral subschemes of $\overline{S}$ which intersect $S$, and is an open subscheme of the Hilbert scheme $\textrm{Hilb}(\overline{S})$ of $(\overline{S}, \mathcal{L})$. We obtain a universal family over $\textrm{Var}(S)$ by restricting the one over $\textrm{Hilb}(\overline{S})$. That there are only finitely many components of $\textrm{Var}(S)$ which parametrize subvarieties of degree at most some fixed integer $d$ is proven in \cite{264428} (cf. \cite[Lem. 5.9]{zbMATH07673321}).

\section{Intersections of subvarieties outside the B\'{e}zout range}\label{sec3}

In this section we prove the following, which contains \Cref{thmA} and the second part of \autoref{mainthmB} (2):

\begin{thm}
\label{abthmsimplecase}
Let $A$ be a simple abelian variety and $X,Y$ irreducible Zariski closed subvarieties such that $\dim X+\dim Y \geq  \dim A$. Then 
\begin{enumerate}
\item[(1)] For all but finitely many $\lambda \in  \operatorname{End} (A)$, $X $ and $\lambda Y$ meet properly, in the sense of \Cref{meetproperly}.
\item[(2)] For any infinite subset $\{ n_{i} \}_{i \in \N} \subset \mathbb{Z} \subset  \operatorname{End}(A)$, the set
\begin{displaymath}
\bigcup_{i \in \N  } (X \cap [n_{i}] Y)^{\operatorname{e}}
\end{displaymath}
is analytically dense in $X$.
\end{enumerate}
\end{thm}

Part (1) will serve as a warm up for the later, more technically demanding, uniform version announced in the introduction.

\subsection{Proof of \autoref{abthmsimplecase}(1)}\label{beginiingabbez}

For $A$ a simple abelian variety, we denote by $\Gamma_{\lambda}\subset A \times A$ the graph of a non-zero endomorphism $\lambda : A \to A$ (which, since $A$ is simple, is an isogeny). If $\lambda, \lambda' : A \to A$ are isogenies, we write $\Gamma_{\lambda, \lambda'} \subset A \times A$ for the image of $A \xrightarrow{(\lambda, \lambda')} A \times A$. We start with three preparatory lemmas and we freely use the special and weakly special vocabulary reviewed in \Cref{typvsatyp}. 

\begin{lem}\label{intersectionlamb}
Let $(t, t') \in A \times A$ be a point and $\lambda_{1}, \lambda'_{1}, \lambda_{2}$, and $\lambda'_{2}$ be isogenies $A \to A$. If $$\lambda_{1} \lambda'^{-1}_{1} \neq \lambda_{2} \lambda'^{-1}_{2} \in \operatorname{End}(A)_{\mathbb{Q}},$$ then the set $$\Gamma_{\lambda_{1}, \lambda'_{1}} \cap (\Gamma_{\lambda_{2}, \lambda'_{2}} + (t, t'))$$ is finite.
\end{lem}

\begin{proof}
As a consequence of \cite[Prop. 5.1]{zbMATH05934826} the intersection of two weakly special subvarieties of an abelian variety has weakly special components. Moreover a weakly special subvariety of $A \times A$ projects onto a weakly special subvariety of $A$ under either projection $A \times A \to A$. Because $A$ is simple, there are no positive-dimensional abelian subvarieties strictly contained in $A$, and hence no positive-dimensional weakly special subvarieties of $A \times A$ of dimension less than $\dim A$. Thus either:
\begin{itemize}
\item[(i)] the intersection $\Gamma_{\lambda_{1}, \lambda'_{1}} \cap (\Gamma_{\lambda_{2}, \lambda'_{2}} + (t, t'))$ is a finite set of points; or
\item[(ii)] we have $\Gamma_{\lambda_{1}, \lambda'_{1}} = \Gamma_{\lambda_{2}, \lambda'_{2}} + (t, t')$.
\end{itemize}
Because $\Gamma_{\lambda_{2}, \lambda'_{2}}$ contains the identity, the second option occurs only if $(t, t') \in \Gamma_{\lambda_{1}, \lambda'_{1}}$. But then $\Gamma_{\lambda_{1}, \lambda'_{1}}$ is an abelian subvariety so it follows by subtracting that $\Gamma_{\lambda_{2}, \lambda'_{2}} \subset \Gamma_{\lambda_{1}, \lambda'_{1}}$, giving the equality for dimension reasons. Such an equality implies that $\lambda_{1} \lambda'^{-1}_{1} = \lambda_{2} \lambda'^{-1}_{2}\in \operatorname{End}(A)_{\mathbb{Q}}$, meaning only (i) is possible.
\end{proof}

\begin{lem}\label{hodgeab}
If $X,Y$ are irreducible subvarieties of $A$ of positive dimension, then $Y \times X \subset A \times A$ does not lie in any strict weakly special subvariety of $A \times A$.
\end{lem}

\noindent (The positive dimension condition is necessary, since $\{ y \} \times X$ lies in the weakly special subvariety $\{ y \} \times A$ for any choice of $X$.)

\begin{proof}
If $X,Y$ have positive dimension, their product cannot lie in a weakly special subvariety of the form $A\times \{a\}$. Weakly special subvarieties of $A\times A$ that dominate both factors and are strictly contained in $A \times A$ are translates of images of a map $A \xrightarrow{(\lambda, \lambda')} A \times A$, with both $\lambda$ and $\lambda'$ isogenies (cf. for example \cite[Thm. 1.1]{Kani2005_AbelianSubvarietiesShimura}). If $X$ and $Y$ have positive dimension, it is not possible for $Y \times X$ to lie in $\{(\lambda a, \lambda' a) : a \in A \}$, and the same fact persists after translation, proving the claim.
\end{proof}

\noindent From now on we use the expected and unexpected terminology introduced in \Cref{meetproperly}. 
\begin{lem}\label{atypicalprodab}
If $\dim X+\dim Y \geq  \dim A$, for each $\lambda$ the projection to the second factor induces a surjective finite map $( Y\times X \cap \Gamma_\lambda )^{\operatorname{ue}} \to (X\cap \lambda Y)^{\operatorname{ue}}$, and each component $Z$ of the source surjects onto a component $W$ of the target.
\end{lem}
We remark here that, for any $\lambda \in \End(A)$, $\lambda Y$ is irreducible and of dimension $\dim Y$. In particular, if $\dim X+\dim Y \geq  \dim A$, then \Cref{bezab} implies that $X\cap \lambda Y \neq \emptyset$. Moreover the unexpected components of $X\cap \lambda Y$ are necessarily of positive dimension and so are the the ones of $Y\times X \cap \Gamma_\lambda$, by general intersection theory (eventually resting on Krull's principal ideal theorem). The latter fact will be important later to invoke the geometric Zilber-Pink theorem.

\begin{proof} 
Note first that the map $(Y \times X) \cap \Gamma_{\lambda} \to (X \cap \lambda Y)$ is surjective, and also finite since $\Gamma_{\lambda} \to A$ is finite. Let $W$ be a component of $(X \cap \lambda Y)^{\textrm{ue}}$, and consider the locus $Z = \{(y, \lambda y): y \in Y, \lambda y \in W \}$. Note that $Z$ lies inside $Y \times X \cap \Gamma_{\lambda}$ by construction. Then $Z \to W$ is surjective, so it suffices to show that
\[ \codim _{A\times A} (Z) <\codim _{A\times A} (Y \times X) + \codim _{A\times A}( \Gamma_{\lambda}) . \]
Rewriting this inequality more explicitly this means
\begin{displaymath}
 2 \dim A - \dim Z < 2 \dim A - \dim X - \dim Y + \dim A
\end{displaymath}
or, more simply
\begin{displaymath}
\dim Z > \dim X + \dim Y - \dim A.
\end{displaymath}
But $\dim Z = \dim W$, so this is just the condition that $W$ is a component of $(X\cap \lambda Y)^{\text{ue}}$.

It remains to show that every component $Z$ of $( Y\times X \cap \Gamma_\lambda )^{\text{ue}}$ maps to a component of $(X\cap \lambda Y)^{\text{ue}}$. But again using the finiteness of $\Gamma_{\lambda} \to A$ the image $W \subset X \cap \lambda Y$ of $Z$ has the same dimension as $Z$, hence the result.
\end{proof}

We recall the geometric Zilber-Pink theorem for subvarieties of $A \times A$ (proved first by Habegger and Pila in \cite{MR3552014} and also a special case of \Cref{mixedgeomZP}), which says the following:
\begin{thm}
\label{geoZPAA}
Let $Z \subset A \times A$ be an irreducible subvariety that does not lie inside a (strict) weakly special subvariety. Then there exists finitely maps of abelian varieties $\rho_{i} : A \times A \to B_{i}$ such that for any weakly special subvariety $W \subset A \times A$, each component of $[W \cap Z]^{\textrm{ue}}$ lies in the fibre of some $\rho_{i}$. 
\end{thm}

\begin{proof}[Proof of \autoref{abthmsimplecase}(1)]
Since, as a consequence of \Cref{hodgeab}, the variety $Y \times X$ does not lie in a strict weakly special subvariety of $A \times A$, we can invoke \autoref{geoZPAA}, which in this case says the following:
\begin{quote}
There exist finitely many algebraic families of strict weakly special subvarieties of $A \times A$ such that all components of $(Y \times X \cap \Gamma_{\lambda})^{\textrm{ue}}$ for all $i$ lie in the fibre of such a family.
\end{quote}
Because non-trivial weakly special subvarieties of $A \times A$ are all translates of loci of the form $\Gamma_{\beta, \beta'}$ (as explained during the proof of \Cref{hodgeab}), each of these families may be chosen to be of the form $\{ \Gamma_{\beta, \beta'} + (a, a') \}_{(a, a') \in \mathcal{A} \subset A \times A}$ where $\mathcal{A}$ is some algebraic locus on $A \times A$. Any component of $(Y \times X \cap \Gamma_{\lambda})^{\textrm{ue}}$ is necessarily positive dimensional, so applying \autoref{intersectionlamb} such a component can only be contained in some $\Gamma_{\beta, \beta'} + (a,a')$ if in fact $\lambda = \beta \beta'^{-1}$ in $\operatorname{End}(A)_{\mathbb{Q}}$. Since there are finitely many families, this happens for only finitely many $\lambda$.

Applying \autoref{atypicalprodab} to pass between intersecting $Y \times X$ with $\Gamma_{\lambda}$ and intersecting $X$ with $\lambda Y$ then gives the result.
\end{proof}

\subsection{Proof of \autoref{abthmsimplecase}(2)}\label{equidi}

Let $\pi : \mathbb{C}^{g} \to A(\C)$ be the natural uniformization with kernel $\Lambda$. Observe that no analytic germ of the complex analytic spaces $\pi^{-1}(X)$ and $\pi^{-1}(Y)$ is contained in any non-trivial linear subspace of $\mathbb{C}^{g}$, as can be shown from the simplicity of $A$ and the fact that the maps $\pi_{1}(X) \to \pi_{1}(A)$ and $\pi_{1}(Y) \to \pi_{1}(A)$ have image of finite index as a consequence of \cite[Prop. 5.1]{zbMATH05934826}. We view each of the $n_{i}$ in the given sequence as a linear map $[n_{i}] : \mathbb{C}^{g} \to \mathbb{C}^{g}$ which preserves $\Lambda$. We may observe that it suffices to prove the following claim:
\begin{thm}
\label{firstredcprop}
Let $A = \mathbb{C}^{g} / \Lambda$ be a simple abelian variety, and $\{ [n_{i}] \}_{i \in \N} \subset \operatorname{End}(A)$ an infinite subset. For any two smooth locally closed irreducible complex analytic subvarieties $\widetilde{X} \subset \mathbb{C}^{g}$ and $\widetilde{Y} \subset \mathbb{C}^{g}$ with $\dim \widetilde{X} + \dim \widetilde{Y} = g$, not contained in a non-trivial affine linear subspace of $\mathbb{C}^{g}$, there exists $i$ such that $[n_{i}] \widetilde{Y}$ intersects $\widetilde{X}$ up to translation by $\lambda \in \Lambda$. Moreover the intersection can be assumed to be proper.
\end{thm}

\begin{rem}
Because in the above statement we consider \emph{locally closed} subvarieties in the Euclidean topology, one automatically obtains that the intersections $\widetilde{X} \cap ([n_{i}] \widetilde{Y} + \lambda)$ are dense in $\widetilde{X}$, since $\widetilde{X}$ can always be replaced with any of its open subsets.
\end{rem}

\begin{proof}
For ease of notation we now refer to $\widetilde{X}, \widetilde{Y}$ as $X, Y$, respectively. We fix a point $x \in X$. \autoref{findniailem} below uses an equidistribution argument to produce a sequence of points, arising from a fixed $y \in Y$, which converge to $x$ modulo $\Lambda$. Fix $n_{i}, a_{i}$ and $y$ as in \autoref{findniailem}. The tangent bundle $T \mathbb{C}^{g}$ has an obvious trivialization, and moreover each tangent space is naturally identified with an affine linear subspace of $\mathbb{C}^{g}$. We let $L_{x} \subset \mathbb{C}^{g}$ be the affine linear subspace passing through $x$ which is the translate to $x$ of $T_{y} Y$. 

Let $d$ be the dimension of $Y$, and $D$ be the closed complex $d$-dimensional analytic disk of radius $1$, and $U$ a manifold. The set $\textrm{Emb}(D, U)$ of embeddings of $C^{\infty}$-manifolds has a natural ``weak'' topology, defined as in \cite[Ch. 2, \S1]{zbMATH03555096}. Because $D$ is compact, this is also the same as the ``strong'' topology defined in op. cit., as explained on \cite[pg. 35]{zbMATH03555096}. Fix a collection $v_{1}, \hdots, v_{d}$ of vectors in $T_{x} \mathbb{C}^{g}$ which generate $T_{x} L_{x}$. Let $e_{1}, \hdots, e_{d}$ be the standard unit tangent vectors of the tangent space $T_{0} D$. 

Thanks to \autoref{tausigmaseqconstrlem}, we obtain a sequence $(\sigma_{i})_i$ of elements of $\textrm{Emb}(D, U)$ which converges to a parameterization of $L_{x} \cap U'$, with $U' \subset U$ a compact neighbourhood of $x$ in $U$. Moreover each $\sigma_{i}$ parameterizes an open submanifold of a translate by $a_{i} \in \Lambda$ of $n_{i} Y$. Now observe we could have chosen $x$ and $y$ such that $T_{x} L_{x} + T_{x} X = T_{x} \mathbb{C}^{g}$: this is possible by the non-linearity assumption together with the assumption that $X$ and $Y$ have complementary dimension. It now follows from \cite[Ch. 2, Thm 2.1]{zbMATH03555096} that the image of $\sigma_{i}$ intersects $X$ for $i$ large enough, completing the proof.
\end{proof}

\begin{lem}
\label{findniailem}
There exists $y \in Y$ not lying in any non-trivial $\mathbb{Q}$-subspace of $\Lambda \otimes \mathbb{Q}$, such that, after replacing $\{ n_{i} \}_{i \in \N}$ with a subsequence, there are $a_{i} \in \Lambda$ for which $[n_{i}] y + a_{i} \to x$, as $i\to + \infty$.
\end{lem}
(For other applications of related equidistribution techniques to the arithmetic of the Zilber-Pink conjecture we refer to \cite{zbMATH05657090}.)
\begin{proof}
We may identify $\mathbb{C}^{g}$ with $ \mathbb{R}^{2g}$ and, after scaling $\Lambda$, reduce to considering the analogous question for the product $(S^{1})^{2g}$ obtained by identifying the sides of the unit square $[0,1]^{2g}$. It suffices to show that, after replacing $y$ with another point of $Y$, the sequence $n_{i} y$ is equidistributed in $(S^{1})^{2g}$ with respect to the standard Haar measure. Applying \cite[Cor. 1.1.3]{zbMATH06110460} it suffices to prove that the sequences $k \cdot n_{i} y$ are equidistributed in $S^{1}$ for each $k \in \mathbb{Z}^{2g} \setminus \{ 0 \}$. Since $Y$ does not lie inside a $\mathbb{Q}$-linear subspace of $\mathbb{R}^{2g}$ by assumption, the image of $Y$ inside $S^{1}$ under each such projection has positive measure. On the other hand by the metric equidistribution theorem of Weyl \cite{zbMATH02610490} \cite[Thm. 4.3, Ch I]{zbMATH03440485}, the subset of $z \in S^{1}$ for which the sequence $k \cdot n_{i} z$ is not equidistributed has measure zero. It follows that one can choose $y \in Y$ which does not land inside any such subset, and therefore the sequence $n_{i} y$ is equidistributed in $(S^{1})^{2g}$.
\end{proof}

\begin{lem}
\label{tausigmaseqconstrlem}
For any neighbourhood $U \subset \mathbb{C}^{g}$ of $x$, there exists a sequence of embeddings $\tau_{i} : D \to \mathbb{C}^{g}$, sending $0$ to $y$, with the following properties:
\begin{itemize}
\item[-] the image of $D$ is locally closed and contained in $X$;
\item[-] the maps $\sigma_{i} := [n_{i}] \tau_{i} + a_{i}$ give elements of $\textrm{Emb}(D,U)$; 
\item[-] the maps $\sigma_{i}$ converge to a parameterization of $L_{x} \cap U'$, with $U' \subset U$ some compact neighbourhood of $x$, and such that $d\sigma_{i}(e_{j}) \to v_{j}$ for each $j = 1, \hdots, d$.
\end{itemize}
\end{lem}

\begin{proof}
Because $Y$ is smooth at $y$, we can choose a closed neighbourhood $B$ at $y$ small enough so that the pair $(B \cap Y, B)$ is biholomorphic to $(D, V)$ via a biholomorphism $P : V \xrightarrow{\sim} B$, where $V \subset \mathbb{C}^{g}$ is the closed unit ball and $D$ is a closed $d$-dimensional subdisk. Shrinking $B$ if necessary, we may assume that $[n_{1}] P(D) + a_{1} \subset U$. We obtain our first map $\tau_{1}$ as the restriction of $P$ to $D$. 

Let $u_{1i}, \hdots, u_{di}$ be a sequence of vectors which span $T_{y} Y$ and such that $d[n_{i}] u_{ji} \to v_{j}$ for each $j$. We may consider holomorphic maps $\eta_{i} : D \to D$ such that $[dP]_{0} \circ [d\eta_{i}]_{0}$ sends the standard unit vectors in $D$ to $u_{1i}, \hdots, u_{di}$. We then define $\tau_{i} = \tau_{1} \eta_{i}$. We let $Q : T_{0} D \to T_{x} \mathbb{C}^{g}$ be the linear map which sends $e_{j}$ to $v_{j}$ for $j = 1, \hdots, d$, regarded as a map $\mathbb{C}^{g} \to \mathbb{C}^{g}$. Shrinking $B$ if necessary, we may assume that $Q(D) \subset U$. 

That the $\sigma_{i}$ converge to a parameterization of $L_{x}$ to $0$'th and $1$'st order is by construction, so it suffices to argue two things:
\begin{itemize}
\item[(1)] the higher-order derivatives of the $\sigma_{i}$ converge to zero as $i \to \infty$; and
\item[(2)] the maps $[n_{i}] \tau_{i} + a_{i}$ factor through $U$. 
\end{itemize}
As we did not specify any conditions on the higher-order derivatives of the $\eta_{i}$, we are still free to adjust them as necessary to satisfy both conditions. For this purpose write $\eta_{i} = \eta_{i,1} + \eta_{i,2}$, with $\eta_{i,1} : D \to D$ the linear approximation to $\eta_{i}$ at $0$, and likewise write $P = P_{2} + P_{1}$. One then has that
\[ P \circ \eta = (P_{2} + P_{1}) \circ (\eta_{i,2} + \eta_{i,1}) = P_{2} \circ \eta_{i,2} + P_{1} \circ \eta_{i,2} + P_{2} \circ \eta_{i,1} + P_{1} \circ \eta_{i,1} .  \]
We can ensure the terms $P_{2} \circ \eta_{i,2}$ and $P_{1} \circ \eta_{i,2}$ converge to $0$ by taking $\eta_{i,2} \to 0$, and the term $P_{1} \circ \eta_{i,1}$ has already been fixed by the condition on $[dP]_{0} \circ [d\eta_{i}]_{0}$. To show (1), it therefore suffices to show that $P_{2} \circ \eta_{i,1} \to 0$. Since $P_{2}$ is fixed, this follows from the claim that $\eta_{i,1} \to 0$, and this follows from the fact that $n_{i} P_{1} \circ \eta_{i,1}$ is bounded as $n_{i} \to \infty$.

To show (2), observe $Q(D) \subset U$ by construction, and $\sigma_{i} \to Q$. Since $D$ is compact, eventually $(Q - \sigma_{i})(D)$ lies an arbitrarily small neighbourhood $\mathcal{N}$ of $0 \in \mathbb{C}^{g}$. Then using that $Q(D)$ is compact in the open neighbourhood $U$, eventually $Q(D) - \mathcal{N}$ also lies in $U$, hence $\sigma_{i}(D) \subset Q(D) - (Q - \sigma_{i})(D) \subset U$. 
\end{proof}

\begin{rem}
A similar argument can be used to prove the following related statement:
\begin{quote}
Let $A$ be a simple abelian variety, $Z \subset A \times A$ a subvariety of dimension $\dim A$ which intersects a generic subvariety $\{ a \} \times A$ transversely. Then given any infinite set $\{ n_{i} \}_{i \in \N}$ the intersections $Z \cap \Gamma_{n_{i}}$, with $\Gamma_{n_{i}} \subset A \times A$ the graphs of multiplication-by-$n_{i}$, are dense in $Z$. 
\end{quote}
The argument is similar except that one produces a sequence of germs in $\mathbb{C}^{g} \times \mathbb{C}^{g}$ converging to ``vertical'' linear subspaces $\{ a \} \times \mathbb{C}^{g}$ which intersects lifts of germs of $Z$. 
\end{rem}

\section{Atypical intersections and uniformity}\label{sec4}

In this section, assuming background from the theory of variations of mixed Hodge structures, we prove the uniform geometric Zilber-Pink conjecture for integral admissible graded-polarizable integral variation of mixed Hodge structures (VMHS from now on). The main result of the section, which we believe to be of independent interest, is \Cref{rhounigeoZP}---a special case of which will be needed for the uniform part of  \Cref{mainthmB} (1), as we will show in \Cref{sec5}. 

Readers interested only in the case of abelian varieties may replace the arbitrary VMHS appearing below by the special case arising from the universal Poincar\'e torsor, as constructed in \S\ref{periodtorsdefforA}. The necessity of working with families of abelian varieties, rather than with a fixed abelian variety $A$, stems from our requirement that the constant $C$ appearing in \Cref{mainthmB}(1) depends only on $\dim A$. 

We freely use the language of our earlier paper \cite{2024arXiv240616628B}; see in particular \S4 in \emph{op. cit.}.
 See also \cite[\S4]{zbMATH07206871} for an introduction to pure and mixed Hodge theory and its intimate relationship with families of abelian varieties. 

\subsection{Proof of the uniform geometric Zilber-Pink}
\label{unifgeozpsec}

We give an uniform version of our geometric mixed Zilber-Pink \cite[Thm. 1.10 and Thm. 7.1]{2024arXiv240616628B} which we recall here (proven first in \cite{2021arXiv210708838B}, in the pure case):

\begin{thm}\label{mixedgeomZP}
Let $(S,\V)$ be a VMHS. There are only finitely many families of maximal monodromically atypical subvarieties of $S$ for $\V$.
\end{thm}

 We fix $\overline{S}$ a compactification of $S$ and an ample line bundle $\mathcal{L}$. Using this we can speak of degrees of subvarieties of $S$ (in the sense of \autoref{degreesubdef}), and consider a universal family 
\begin{displaymath}
\widetilde{f} : \widetilde{\mathcal{X}} \to \textrm{Var}(S)
\end{displaymath} 
  of subvarieties of $S$ as described at the end of \S\ref{sec:deg}. The base $\textrm{Var}(S)$ has countably infinitely many connected components, each of finite type over $\mathbb{C}$, and for each $d$ there exists a finite subset of these components whose union is $\textrm{Var}(S)_{d}$ such that every subvariety of degree at most $d$ appears as a fibre over the components in this subset. When we speak of a ``universal family'' of subvarieties of degree at most $d$ in $S$, we mean to fix such a family $f : \mathcal{X} \to \mathcal{H} := \textrm{Var}(S)_{d}$ obtained by restricting $\widetilde{f}$. 

Let $(S,\V)$ be an admissible graded-polarizable $\Z$VMHS, and fix an algebraic map $\rho : S \to T$.

\begin{defn}
\label{rhofamdef}
A $\rho$-\emph{family of subvarieties of} $S$ is a pair of maps $E \xleftarrow{a} Y \xrightarrow{\iota} S$ of algebraic varieties such that for each closed point $e \in E$ the map $\iota_{e} : Y_{e}=a^{-1}(e) \to S$ induced by $\iota$ is a closed embedding into some fibre of $\rho$. 
\end{defn}

\begin{defn}
\label{rhowspdef}
A subvariety $Y \subset S$ is called $\rho$-\emph{weakly special} if it is contained in some fibre $S_{t}$ of $\rho$ and is weakly special for the  $\Z$VMHS $(S_{t}, \restr{\V}{S_{t}})$. It is said to be $\rho$-\emph{(monodromically) atypical} if it is $\rho$-(monodromically) atypical for the $\Z$VMHS $(S_{t}, \restr{\V}{S_{t}})$. 
\end{defn}

The main result of the section is the following; its proof refines and complements the one of \Cref{mixedgeomZP}.
\begin{thm}[Uniform geometric Zilber-Pink for VMHS]
\label{rhounigeoZP}
There exists a set of finitely many $\rho$-families of subvarieties of $S$ whose fibres are exactly the maximal monodromically atypical $\rho$-weakly special subvarieties of $S$.
\end{thm}

Note that one obtains as a corollary:

\begin{cor}
\label{rhounigeoZPdegver}
Let $(S,\V)$ be an admissible graded-polarizable $\Z$VMHS. Then there exists $d = d(b)$ with the following property: for any subvariety $X \subset S$ of degree $\leq b$, the maximal monodromically weakly special subvarieties of $(X, \restr{\V}{X})$ have degree at most $d$.
\end{cor}

\begin{proof}[Proof of \Cref{rhounigeoZPdegver}]
Let $f : \mathcal{X} \to \mathcal{H}$ be a universal family of subvarieties of degree at most $b$. Then this family comes with a map $\mathcal{X} \to S$, and pulling back $\V$ under this map we may apply \autoref{rhounigeoZP} with $\rho = f$. Then the fibres of the finitely many algebraic families of subvarieties given to us by \autoref{rhounigeoZP} will have fibres of degree bounded by some integer $d$.
\end{proof}

\begin{proof}[Proof of \Cref{rhounigeoZP}]
In what follows we write ``atypical'' for ``monodromically atypical'', since non-monodromic notions of atypicality will not arise.

The first step is to introduce the \emph{period torsor} $P$. The construction proceeds entirely analogously to \S\ref{periodtorsdefforA} as well as \cite[\S2.2]{2024arXiv240616628B} and appearing also in \cite{2022arXiv220805182B}. Let $(\mathcal{V}, \nabla)$ be the flat algebraic vector bundle associated to $\V$ by the Riemann-Hilbert correspondence. We then consider $\mathcal{I} = \textrm{Iso}(\mathcal{V}_{a'}, \mathcal{V}_{a})$ as in loc. cit., and define $P$ as the algebraic closure of the flat leaf corresponding to $\textrm{id}_{a}$, as in loc. cit. Then again by \cite[\S2.2]{2024arXiv240616628B} we find that $P$ is a torsor for the group
\begin{align*}
\mathbf{H} := \overline{\textrm{im}[\pi_{1}(S,a) \to \GL(\mathbb{V}_{a}) \simeq \GL(\mathcal{V}_{a})]}^{\textrm{Zar}} 
\end{align*}
obtained as the Zariski closure of the image of monodromy (which we can assume to be connected, simply by passing to a finite covering of the base). In fact $P$ comes with a $\mathbf{H} (\C)$ flat principal connection, inducing a foliation on $P$.

We make some preliminary reductions. In particular, we may assume that:
\begin{quote}
The fibres of $\rho$ are smooth, connected, all of the same dimension, and their algebraic monodromy groups agree with a fixed $\mathbb{Q}$-normal subgroup $\mathbf{N}$ of $\mathbf{H}$.
\end{quote}
This we can argue by Noetherian induction, progressively shrinking $S$ (either directly or by shrinking $T$) and restricting $(\mathbb{V}, P)$. Because the dimension (resp. number of irreducible components) of the fibres of $\rho$ is a constructible condition on $T$ \cite[\href{https://stacks.math.columbia.edu/tag/05F9}{Lemma 05F9}]{stacks-project} (resp. \cite[\href{https://stacks.math.columbia.edu/tag/055B}{Lemma 055B}]{stacks-project}), we may use Noetherian induction on $T$ to assume the dimension condition and that all fibres have the same number of irreducible components. Similarly, using Noetherian induction on $S$ we may replace $S$ with an open subvariety so that the fibres of $S \to T$ are smooth and irreducible (so in particular connected). Then using \cite[Cor. 5.1]{zbMATH03520733} and Noetherian induction on $T$ again, one can assume that $S \to T$ is a topological fibration. It follows from the homotopy exact sequence 
\[ \pi_{1}(S_{t}) \to \pi_{1}(S) \to \pi_{1}(T) \to 1 \]
that the algebraic monodromy group $\mathbf{H}_{t}$ of each fibre $S_{t}$ is normal in $\mathbf{H}$. Changing $t$ amounts to conjugating $\mathbf{H}_{t}$, so $\mathbf{N} := \mathbf{H}_{t}$ is independent of $t$.

Let $\ch{L}$ be the flag variety of all flags on $\mathcal{V}_{a}$ of the same type as $F^{\bullet}_{a}$, and let $\ch{D}$ be the orbit of $F^{\bullet}_{a}$ under $\mathbf{H}$ in $\ch{L}$. Then $\ch{D}$ is a projective algebraic variety and we have a natural map $r : P \to \ch{D}$, constructed as in \cite[\S4.6]{2024arXiv240616628B}. We then have, as proven in \cite[Lem. 7.2]{2024arXiv240616628B}, that all varieties $\ch{Q} \subset \ch{D}$ such that $\ch{Q}$ is (the compact dual of) a weakly special subdomain of $\ch{D}$ appear as fibres of some common algebraic family $g : \mathcal{D} \to \mathcal{C}$ over a possibly disconnected base of finite type (note that not all fibres of this family need be weakly special domains). It follows that all weakly special subdomains of $\ch{D}$ which lie inside some orbit of $\mathbf{N}$ also belong to some common algebraic family, which we again denote $g$.

Given a set of subvarieties $\mathcal{S}$ of a fixed variety $X$ we write $\mathcal{I}(\mathcal{S}, X)$ for the set of all varieties obtained by taking intersections of elements of $\mathcal{S}$. We now have the following claim, which is analogous to \cite[Prop. 7.4]{2024arXiv240616628B}.
\begin{quote}
\textbf{Claim:} There exists a family $f : \mathcal{Z} \to \mathcal{Y}$ of subvarieties of the period torsor $P \to S$ such that:
\begin{itemize}
\item[(1)] the fibres of $f$ are exactly the irreducible components of the intersections between
\begin{itemize}
\item[-] varieties in the set $r^{-1}(\mathcal{I}(\mathcal{S}, \ch{D}))$, where $\mathcal{S}$ is the set of fibres of $g$; and
\item[-] the fibres of $P \to T$;
\end{itemize}
and
\item[(2)] all $\rho$-weakly special subvarieties $Y$ of $S$ arise as follows: there is a point $y \in \mathcal{Y}(\mathbb{C})$ such that $Y$ is a component of the projection to $S$ of $\mathcal{Z}_{y} \cap \mathcal{L}$ for some horizontal leaf $\mathcal{L}$ of $P$ above $S$.
\end{itemize}
\end{quote}
As in the proof of \cite[Prop. 7.4]{2024arXiv240616628B}, it suffices to prove (1). To do this, one applies \cite[Lem 7.3]{2024arXiv240616628B} to the family $g \circ r$ constructed above and obtain a family $\mathcal{I}(g \circ r)$ over some space $\mathcal{I}(\mathcal{C})$. We then consider the family $f$ over $T \times \mathcal{I}(\mathcal{C})$ whose fibres are the intersections of the fibres of $P \to T$ and the fibres of $\mathcal{I}(g \circ r)$.

Using the claim, we now mimic the proof in \cite[\S7.2]{2024arXiv240616628B}. The notation $f^{(< k)}$ is defined as in that argument, as are the loci $\mathcal{Y}(e), \mathcal{K}(e)'$ and $\mathcal{K}(e)$. Note that we may regard all points in these loci as triples $(x, t, c)$, with $x$ a point of $P$ and $(t,c)$ a point over the base variety $\mathcal{Y} = T \times \mathcal{I}(\mathcal{C})$. Applying \cite[Prop. 3.10]{2024arXiv240616628B} in the same was as in \S7.2 in loc. cit., one sees these loci are constructible. The rest of the proof proceeds by replacing the term ``weakly special'' with $\rho$-weakly special. 

More precisely, the argument of \cite[\S7.2]{2024arXiv240616628B} proceeds by proving \cite[Prop. 7.7]{2024arXiv240616628B} and \cite[Lem. 7.11]{2024arXiv240616628B}, and the analogues of these statements are given by \autoref{surjectontogermprop} and \autoref{inducepointlem} below. After this, we consider a countable collection $\{ h_{i} : C_{i} \to B_{i} \}_{i=1}^{\infty}$ of families which contain all $\rho$-weakly special subvarieties, and such that all fibres of these families are atypical and $e$-dimensional. This is constructed by:
\begin{itemize}
\item[(1)] first constructing a countable collection $\{ h'_{i} : C_{i} \to B_{i} \}_{i=1}^{\infty}$ of families whose fibres is exactly the set of all weakly special subvarieties of $S$ using \cite[Prop. 4.19]{2024arXiv240616628B};
\item[(2)] then for each $h'_{i}$ constructing an associated family $h_{i}$ whose fibres are the irreducible components of intersections of fibres of $h'_{i}$ with fibres of $\rho$;
\item[(3)] refining these families so that the fibres of each $h_{i}$ all have the same dimension; 
\item[(4)] throwing away those families not corresponding to subvarieties of dimension $e$;
\item[(5)] iteratively using \cite[Cor. 5.1]{zbMATH03520733} on the bases $B_{i}$ so that the fibres of each $h_{i}$ have a common algebraic monodromy group; and
\item[(6)] throwing away those families which don't correspond to families of $\rho$-atypical varieties.
\end{itemize}
Note that after step (5), one fibre of some family $h_{i}$ is $\rho$-atypical if and only if they all are.

To conclude, we observe that \autoref{surjectontogermprop} and \autoref{inducepointlem} (proven below) together give $\mathcal{Y}(e) \setminus \mathcal{K}(e) = \bigcup_{i=1}^{\infty} \mathcal{Y}(f^{(< \nu(e))}, e, h_{i})$, where $\nu(e) = e + \dim (\mathbf{N} \cdot F^{\bullet}_{a}) - \dim \rho^{-1}(t)$, where we note that in our setup $\dim \rho^{-1}(t)$ is independent of $t$. Applying \cite[Lem. 3.13]{2024arXiv240616628B} we conclude that there is some $n$ such that $\bigcup_{i=1}^{\infty} \mathcal{Y}(f^{(< \nu(e))}, e, h_{i}) = \bigcup_{i=1}^{n} \mathcal{Y}(f^{(< \nu(e))}, e, h_{i})$, which gives the desired result.
\end{proof}

\begin{prop}
\label{surjectontogermprop}
For each $(x,t,c) \in \mathcal{Y}(e) \setminus \mathcal{K}(e)$, let $U$ be a component at $x$ of $\mathcal{Z}_{y} \cap \mathcal{L}_{x}$ of dimension $\geq e$. Then the projection of $U$ to $S$ surjects onto the germ of a strict $\rho$-weakly special subvariety $Y$ of $S$ (of dimension $e$). 
\end{prop}

\begin{proof}
Note that, by construction, the variety $\mathcal{Z}_{y}$ lies inside the torsor $P_{t} \to S_{t}$, where we take the fibres over $t$. Therefore the intersection $\mathcal{Z}_{y} \cap \mathcal{L}_{x}$ can be regarded inside $P_{t}$ and where the leaf $\mathcal{L}_{x}$ is the leaf through $x \in P_{t}$ for the induced flat connection on $P_{t}$. If we then replace the loci $\mathcal{Y}(e), \mathcal{K}(e)', \mathcal{K}(e)$, etc., with the corresponding loci $\mathcal{Y}(e,t), \mathcal{K}(e,t)', \mathcal{K}(e,t)$ which are defined in the same way except using the family $f_{t}$ obtained by restricting $f$ to $\{ t \} \times \mathcal{I}(\mathcal{C})$, then we are exactly in the setup of \cite[\S7.2]{2024arXiv240616628B} and the proof proceeds in the same way. The conclusion then shows that $U$ projects onto a weakly special subvariety of $S_{t}$ of dimension $e$, hence a $\rho$-weakly special with the same properties.
\end{proof}

\begin{lem}
\label{inducepointlem}
Each maximal (among atypical) atypical $\rho$-weakly special $Y$ of dimension $e$ induces a point of $\mathcal{Y}(e) \setminus \mathcal{K}(e)$. 
\end{lem}

\begin{proof}
Since $Y$ is $\rho$-weakly special, it lies in some fibre $S_{t}$ of $S \to T$. Using the same procedure as in \autoref{surjectontogermprop} to define $P_{t}$ and the subloci $\mathcal{Y}(e,t), \mathcal{K}(e,t)', \mathcal{K}(e,t)$, it suffices to show that $Y$ induces a point of $\mathcal{Y}(e,t) \setminus \mathcal{K}(e,t)$. We are then exactly in the setup of \cite[Lem. 7.11]{2024arXiv240616628B} and the same argument completes the proof.
\end{proof}

\subsection{End of the proof of \Cref{mainthmB} (1)}\label{sec5}

Let $\pi : \mathbb{A}_{g,b} \to \mathcal{A}_{g,b}$ be the universal family of abelian varieties equipped with a polarization of degree $b$. (Such a family exists on the level of stacks, and, as we could always replace our base $\mathcal{A}_{g,b}$ by a finite cover by adding level structure, we will commit the abuse of treating $\pi$ as an honest algebraic family.) The universal family $\pi$ admits a universal polarization, giving rise to a relative embedding (cf. \cite[Thm. 5.5]{zbMATH06665024}) $\iota : \mathbb{A}_{g,b} \hookrightarrow \mathbb{P}^{N}_{\mathcal{A}_{g,b}}$ for some $N$. We consider the universal family $f : M \to T$ of all degree $\leq d$ subvarieties of $\mathbb{P}^{N}$ and consider the induced family $\mathcal{f} : \mathcal{M} \to T \times \mathcal{A}_{g,b}$ of subvarieties of the fibres of $\pi$ obtained as intersections of between fibres of $f$ and fibres of $\pi$. After possibly stratifying the base $T \times \mathcal{A}_{g,b}$ and removing fibres with several components, we may assume that the family $\mathcal{f}$ contains as its fibres all irreducible subvarieties $X \subset \mathbb{A}_{g,b}$ which lie in a fibre of $\pi$ and have have degree at most $d$ for the polarization on that fibre. Using two copies of such a family, we may then construct a family $\mathcal{f}_{2} : \mathcal{M}_{2} \to T_{2}$ whose fibres consist of all products $Y \times X \subset \mathbb{A}_{g,b} \times_{\mathcal{A}_{g,b}} \mathbb{A}_{g,b}$ which lie in some fibre of $\mathbb{A}_{g,b} \times_{\mathcal{A}_{g,b}} \mathbb{A}_{g,b} \to \mathcal{A}_{g,b}$ and such that $Y$ and $X$ have degree at most some integer $d$. 

We take $\rho = \mathcal{f}_2$ and apply \autoref{rhounigeoZP}, where the variation of mixed Hodge structure is the one pulled back from $\mathbb{A}_{g,b} \times_{\mathcal{A}_{g,b}} \mathbb{A}_{g,b}$; the local system on $\mathbb{A}_{g,b} \times_{\mathcal{A}_{g,b}} \mathbb{A}_{g,b}$ is constructed from a Poincar\'e bundle over $\mathbb{A}_{g,b} \times_{\mathcal{A}_{g,b}} \mathbb{A}_{g,b}$ via a procedure analogous to the one in \S\ref{periodtorsdefforA}. We then learn that there are finitely many $\rho$-families of weakly special subvarieties of $\mathcal{M}_{2}$ such that all atypical weakly special intersections between $Y \times X \subset A \times A$, with $A$ some fixed fibre of $\pi$, arise as a fibre of one of these families. After using \cite[Cor. 5.1]{zbMATH03520733} we may assume each such family is a topological fibration. Moreover each such family induces at most finitely many families of weakly special subvarieties of $Y \times X$, which we may also assume are topological fibrations, and where the number of such families is bounded by a uniform constant $C$.

Let $g$ be one of the families of weakly special subvarieties of $Y \times X$ just mentioned. Since $g$ is a topological fibration, the fibres of $g$ induce a common conjugacy class of subgroups of $\pi_{1}(A \times A)$, and hence a unique such subgroup since $\pi_{1}(A \times A)$ is abelian. Moreover if such a fibre $F$ lies in a non-trivial abelian subvariety $A'$ of $A \times A$, the simplicity assumption on $A$ implies that the image of $\pi_{1}(F) \otimes_{\mathbb{Z}} \mathbb{Q}$ in $\pi_{1}(A \times A) \otimes_{\mathbb{Z}} \mathbb{Q}$ agrees with the image of $\pi_{1}(A') \otimes_{\mathbb{Z}} \mathbb{Q}$. For a fixed abelian variety $B$, different connected abelian subvarieties of $B$ induce different subgroups of $\pi_{1}(B) \otimes_{\mathbb{Z}} \mathbb{Q}$, so it follows that there are at most $C$ abelian subvarieties of $A \times A$ that intersect $Y \times X$ atypically, where $C$ is a uniform constant. By following again the proof of \autoref{abthmsimplecase}(1) one observes that this $C$ corresponds exactly to the finitely many $\lambda \in \operatorname{End}(A)$ referred to in \autoref{mainthmB} (1), and hence proves the first part of that statement, as desired. \hfill \qedsymbol{} 

\medskip

\noindent \Cref{uniformityinabeliansetting} is deduced from \Cref{rhounigeoZP} in the same way. \hfill \qedsymbol{}

\section{Intersections of subvarieties outside the B\'ezout range}\label{finalsec}

In this final section we use o-minimality and ideas from Diophantine geometry to prove \autoref{densityofnonintthm}. See also \cite{zbMATH05292756, zbMATH06022187, zbMATH07516542} for results related to this circle of ideas. Once again we begin by rephrasing the question as a question about $Y \times X \subset A \times A$: it is equivalent to show that the set
\[ \{ r \leq n : Y \times X \cap \Gamma_{[n]} = \varnothing \} \]
has asymptotic $1 - \delta$. Write $\pi : \mathbb{C}^{g} \to A(\mathbb{C})$ for the uniformization and $\pi_2 : \mathbb{C}^{2g} \to (A \times A)(\mathbb{C})$ for the product map. Fix a fundamental parallelogram $\mathcal{F}$ for the covering action of $\pi$, and let $\widetilde{Y}$ (resp. $\widetilde{X}$) denote the intersection of $\pi^{-1}(Y)$ (resp. $\pi^{-1}(X)$) with $\mathcal{F}$. Letting $$e_{1}, \hdots, e_{2g}$$ be the vectors in $\mathbb{C}^{g} = \mathbb{R}^{2g}$, emmanating from the origin, which correspond to the sides of $\mathcal{F}$, we have that 
\begin{equation}
\label{Fdesc}
\mathcal{F} = \left\{ \mathbf{x} \in \mathbb{C}^{g} : \mathbf{x} = \sum_{i=1}^{2g} a_{i} e_{i}; \hspace{1em} a_{i} \in [0,1) \textrm{ for all }1 \leq i \leq 2g \right\}.
\end{equation}
We use the theory of definable o-minimal structures from \cite{MR1633348}, and importantly that the restriction of the uniformization map $\pi \times \pi$ to $\mathcal{F} \times \mathcal{F}$ is definable in the o-minimal structure $\mathbb{R}_{\textrm{an}}$. From now on, when we say \emph{definable} we mean definable it the o-minimal structure $\mathbb{R}_{\textrm{an}}$.

We can also consider the intersection $\mathcal{I}_{n}$ of $\pi^{-1}_{2}(\Gamma_{[n]})$ with $\mathcal{F} \times \mathcal{F}$. A priori the locus $\pi^{-1}_{2}(\Gamma_{[n]})$ is the union of infinitely many affine-linear subspaces of $\mathbb{C}^{g} \times \mathbb{C}^{g}$, but only finitely many of these intersect $\mathcal{F} \times \mathcal{F}$. Write $L_{n} \subset \mathbb{C}^{g} \times \mathbb{C}^{g}$ for the graph $(z, nz)$ of multiplication-by-$n$; we will use the same notation even when $n$ lies in $\R^\times$ or even $ \mathbb{C}^{\times}$. 

\begin{defn}
 A component of $\mathcal{I}_{n}$ is called a $\Gamma_{[n]}$-\emph{segment}.
\end{defn}

\begin{lem}
\label{Toplyinn}
The number of $\Gamma_{[n]}$-segments is $O(\operatorname{poly}(n))$. Each segement is the intersection with $\mathcal{F} \times \mathcal{F}$ of an affine subspace of $\mathbb{C}^{g} \times \mathbb{C}^{g}$ of the form $L_{n, \mathbf{a}} := L_{n} + (0, \mathbf{a})$, where $\mathbf{a} \in \mathbb{Z}^{2g}$ if we work in the coordinates determined by the basis $e_{1}, \hdots, e_{2g}$. Write
\[ \mathcal{S} = \mathbb{R}^{\times} \times \operatorname{span}_{\mathbb{R}} \{ e_{1}, \hdots, e_{2g} \} \]
for the parameter space of $2g$-dimensional affine subspaces of the $4g$-dimensional real vector space $\mathbb{C}^{2g}$ which are of the form $L_{n} + (0, \mathbf{a})$ of some $L_{n}$, with $n \in \mathbb{R}^{\times}$ and $a \in \mathbb{R}^{2g}$. Then a point of $\mathcal{S}$ corresponding to a segment of $\Gamma_{[n]}$ has height $O(\operatorname{poly}(n))$ (using the na\"ive height on points of $\mathbb{Q}^{2g+1} \subset \mathbb{R}^{2g+1}$). 
\end{lem}

\begin{proof}
That each segment can be written as an intersection with $L_{n} + (0, \mathbf{a})$ with $\mathbf{a}$ integral is from the definitions. In particular, each component of $\pi^{-1}_{2}(\Gamma_{[n]})$ is of the form $L_{n} + (\lambda_{1}, \lambda_{2})$ for $\lambda_{i} \in \Lambda$, where $A = \mathbb{C}^{g} / \Lambda$. In the chosen basis, each $(\lambda_{1}, \lambda_{2})$ is of the form $(\mathbf{a}_{1}, \mathbf{a}_{2})$ where $(\mathbf{a}_{1}, \mathbf{a}_{2}) \in \mathbb{Z}^{4g}$. 
 But the subspace $L_{n} + (\mathbf{a}_{1}, \mathbf{a}_{2})$ agrees with $L_{n} + (0, \mathbf{a}_{2} + n \mathbf{a}_{1})$.

The number of translates of $L_{n}$ that intersect $\mathcal{F} \times \mathcal{F}$ is then the number of integral vectors $\mathbf{a}$ for which there exists $\mathbf{x} \in [0,1)^{2g}$ such that $n \mathbf{x} + \mathbf{a} \in [0,1)^{2g}$. This condition ensures that the entries of $\mathbf{a}$ are non-positive with magnitude at most $n$, hence there are at most $n^{2g}$ possibilities. The same analysis shows that the height of each point of $\mathcal{S}$ corresponding to a segment is similarly bounded by $n$.
\end{proof}

Let us write 
\[ f : T \to S = \mathbb{G}_{m} \times \mathbb{A}^{g} \]
for the algebraic family of affine linear subspaces of $\mathbb{A}^{2g}_\C$ whose fibre above $(n, \mathbf{a}) \in \mathbb{C}^{\times} \times \mathbb{C}^{g} = \mathbb{G}_{m}(\mathbb{C}) \times \mathbb{A}^{g}(\mathbb{C})$ is $L_{n} + (0, \mathbf{a})$. Then $\mathcal{S}$ may be identified with $\mathbb{G}_{m}(\mathbb{R}) \times \mathbb{A}^{g}(\mathbb{C})$. Let $\kappa : T \to \mathbb{A}^{2g}$ be the natural projection which sends each complex fibre of $f$ to the corresponding affine linear subspace of $\mathbb{C}^{2g}$; this map restricts to a closed embedding on fibres of $f$. Let $\mathcal{T} \subset T(\mathbb{C})$ be the intersection $\kappa^{-1}(\widetilde{Y} \times \widetilde{X}) \cap f^{-1}(\mathcal{S})$. 

One can regard $\mathcal{T}$ as a total space for the families of intersections $L_{n,\mathbf{a}} \cap \widetilde{Y} \times \widetilde{X}$ as $(n, \mathbf{a})$ varies in $\mathcal{S}$. It follows from definable cell decomposition \cite[Ch. 3]{MR1633348} that there exists a finite definable cover $\mathcal{T} = \bigcup_{i=1}^{\ell} \mathcal{T}_{i}$ such that the non-empty fibres of the families $\mathcal{T}_{i} \to \mathcal{S}$ are irreducible complex analytic components of $L_{n,\mathbf{a}} \cap \widetilde{Y} \times \widetilde{X}$, and every such component occurs in the fibre of some family $\mathcal{T}_{i} \to \mathcal{S}$. After further refining the $\mathcal{T}_{i}$, we can assume that all of the fibres of each map $\mathcal{T}_{i} \to \mathcal{S}$ have the same dimension. After replacing $\mathcal{S}$ with a possibly disconnected base which is the disjoint union of images of the maps $\mathcal{T}_{i} \to \mathcal{S}$, we can assume we have two such families:
\begin{displaymath}
f_{> 0} : \mathcal{T}_{> 0} \to \mathcal{S}_{> 0}, \text{  and  } f_{0} : \mathcal{T}_{0} \to \mathcal{S}_{0},
\end{displaymath}
which parameterize those components with dimension $> 0$ and dimension $0$, respectively. Note that we have maps 
\begin{displaymath}
\mu_{> 0} : \mathcal{S}_{> 0} \to \mathcal{S}  \text{  and  } \mu_{0} : \mathcal{S}_{0} \to \mathcal{S},
\end{displaymath}
so each point in $\mathcal{S}_{> 0}$ (resp. $\mathcal{S}_{0}$) is associated to some $(n, \mathbf{a})$. Note also that each fibre of $f_{0}$ is just a point of $\widetilde{Y} \times \widetilde{X}$, and $f_{0}$ is invertible on its image.

Let $\rho_{i}$ with $i = 1, \hdots, \ell$ be the finitely many families from \autoref{geoZPAA}, and throw away those $\rho_{i}$ for which the fibres are zero dimensional. We write $\mathcal{R}_{0} \subset \mathcal{S}_{0}$ for the subset with the property
\begin{equation}
\label{R0constr}
\mathcal{R}_{0} = \left\{ s \in \mathcal{S}_{0} : (n, \mathbf{a}) = \mu_{0}(s), \hspace{1em} \begin{array}{c}\dim [ L_{n,\mathbf{a}} \cap \mathcal{F} \times \mathcal{F} \cap \pi^{-1}_{2}(\rho^{-1}_{i}(b)) ] \leq 0, \\ \textrm{for all }1 \leq i \leq \ell \textrm{ and }b \in B_{i} \end{array} \right\}. 
\end{equation}
Note that $\mathcal{R}_{0}$ is a definable subset. For any subset $D$ of a real affine space $\mathbb{R}^{m}$ and any real number $T\geq 1$, we write $D(\mathbb{Q}, T)$ for the subset $D \cap \mathbb{Q}^{m}$ with height at most $T$, where the height is the naive height on rational points in $\mathbb{R}^{m}$. We also refer to $\mathcal{S}_{0}(\mathbb{Q}, T)$ and $\mathcal{R}_{0}(\mathbb{Q}, T)$ by regarding each such set inside (finitely many) real affine spaces, which one can do by mapping each component of $\mathcal{S}_{0}$ (resp. $\mathcal{R}_{0}$) to $\mathcal{S}$ and using the affine space associated to $\mathcal{S}$.  
 
\begin{lem}
\label{fewlineswhichintlem}
For any $\ep > 0$, the set $\mathcal{R}_{0}(\mathbb{Q}, T)$ contains $O(T^{\ep})$ points $r_{0}$ such that
\begin{itemize}
\item[(1)] the point $\mu_{0}(r_{0}) = (n, \mathbf{a})$ defines a linear subspace $L_{n, \mathbf{a}}$ with $n \in \mathbb{Z}_{\leq T}, \mathbf{a} \in \mathbb{Z}^{2g}$; and
\item[(2)] the fibre $f^{-1}_{0}(r_{0}) \subset \widetilde{Y} \times \widetilde{X}$ does not project to a torsion point $A \times A$.
\end{itemize}
\end{lem}
\noindent In what follows we write $\mathcal{P}_{0}(T) \subset \mathcal{R}_{0}(\mathbb{Q}, T)$ for the subset of points satisfying (1) and (2), and set $\mathcal{P}_{0} = \bigcup_{T} \mathcal{P}_{0}(T)$. 

\begin{lem}
\label{P0injlem}
The map $\mathcal{P}_{0} \xrightarrow{f^{-1}_{0}} \mathcal{T}_{0} \to \mathcal{T} \xrightarrow{\kappa} \mathbb{C}^{g} \times \mathbb{C}^{g}$ is injective.
\end{lem}

\begin{proof}
Pick two distinct points $r_{0}, r'_{0} \in \mathcal{P}_{0}$. Set $(n, \mathbf{a}) = \mu_{0}(r_{0}), (n', \mathbf{a}') = \mu_{0}(r'_{0})$, and let $p, p' \in \widetilde{Y} \times \widetilde{X}$ be their images under the map in the statement. Suppose first that $(n, \mathbf{a}) = (n', \mathbf{a}')$. Then if $p = p'$ we must have $r_{0} = r'_{0}$ since each $0$-dimensional component of $L_{n,\mathbf{a}} \cap \widetilde{Y} \times \widetilde{X}$ appears exactly once as a fibre of the family $f_{0}$ by construction. Knowing now that $(n, \mathbf{a}) \neq (n', \mathbf{a}')$, we may observe that if $n = n'$ then $L_{n,\mathbf{a}} \cap L_{n',\mathbf{a}'} = \varnothing$ since then the affine linear subspaces are parallel and do not intersect unless $\mathbf{a} = \mathbf{a}'$. Having reduced to the case where $n \neq n'$, we deduce that if $p = p'$ then $\pi_{2}(p)$ lies inside both $\Gamma_{[n]}$ and $\Gamma_{[n']}$ and is therefore torsion, violating condition (2). Hence $p \neq p'$.
\end{proof}

The proof of uses various results about \emph{functional transcendence} for $A\times A$, due to Ax \cite{zbMATH03407823}. For a general overview of the recent developments and applications we refer also to \cite{2025arXiv250203071B}.

\begin{proof}[Proof of \Cref{fewlineswhichintlem}]
Suppose otherwise. Then there exist some $c, \ep > 0$ and an infinite sequence $\{ T_{i} \}_{i \geq 1}$ such that the number of points of $\mathcal{R}_{0}(\mathbb{Q}, T_{i})$ satisfying (1) and (2) is $\geq c T_{i}^{\ep}$ for all $i \geq 1$. By \cite[Thm. 5.3]{MR2511202} (a generalization of the Pila-Wilkie theorem from \cite{pilawilkie}) one has that the subset $\mathcal{R}_{0}(\mathbb{Q}, T)$ of the definable set $\mathcal{R}_{0}$ lies inside $O(T^{\delta})$ semialgebraic blocks for any $\delta > 0$ (for the notion of semialgebraic block see \cite[Def. 3.2]{MR2511202}). Using \autoref{Toplyinn} and choosing $\delta < \ep$ we can assume that many of these blocks contain at least two distinct points $r_{0}, r'_{0} \in \mathcal{P}_{0}$. Let $B$ be such a block, and let $p, p' \in \widetilde{Y} \times \widetilde{X}$ be their images in $\mathcal{F} \times \mathcal{F}$. By \autoref{P0injlem} we have $p \neq p'$.

Now $B$ contains a real semialgebraic $C$ passing through both $r_{0}$ and $r'_{0}$. Identify both $B$ and $C$ with their images in $\mathcal{S}$. Recall that $\mathcal{S}$ is a real-algebraic locus in the complex points $\mathbb{C}^{\times} \times \mathbb{C}^{g}$ of the complex algebraic variety $S = \mathbb{G}_{m} \times \mathbb{A}^{g}$, and we have an algebraic map $f : T \to S$. Write $f_{\mathbb{R}} : \textrm{Res}_{\mathbb{C}/\mathbb{R}} T \to \textrm{Res}_{\mathbb{C}/\mathbb{R}} S$ for the Weil restriction, and $f_{C} : T_{C} \to C$ for the base-change of this Weil restriction to $C$. Scalar extending, we obtain a complex algebraic family $f_{C,\mathbb{C}} : T_{C,\mathbb{C}} \to C_{\mathbb{C}}$. Using the adjunction between Weil restriction and base-change we have a map
\[ \nu : T \to (\textrm{Res}_{\mathbb{C}/\mathbb{R}} T)_{\mathbb{C}} \]
and so we can consider the family $h : \nu^{-1}(T_{C,\mathbb{C}}) =: \mathcal{W} \to C_{\mathbb{C}}$. By construction this is a complex algebraic family of affine-linear subspaces of $\mathbb{A}^{2g}$ over a complex curve $C_{\mathbb{C}}$ which includes those parameterized by $C$. 

Write $W$ for the image of $\mathcal{W}$ in $\mathbb{A}^{2g}$. This is a constructible complex algebraic set of dimension $g+1$ whose intersection with $\widetilde{Y} \times \widetilde{X}$ contains a positive-dimensional complex-analytic component $V$ containing the points $p$ and $p'$ (such a component must exist because $W \cap \widetilde{Y} \times \widetilde{X}$ contains the image of $f^{-1}_{0}(C)$, which is a connected definable locus containing two distinct points $p$ and $p'$). Since, by assumption, $\dim X + \dim Y < \dim A$, we have that $$\dim (\widetilde{Y} \times \widetilde{X}) - \dim V < 2 \dim A - \dim W.$$ Therefore the Ax-Schanuel theorem \cite{zbMATH03407823} tells us that this locus $V$ projects under $\pi \times \pi$ into a positive-dimensional weakly special intersection strictly contained in $Y \times X$, hence necessarily the fibre of some $\rho_{i}$. Let $F$ be this fibre, then $V$ also lies inside some affine linear sublocus $\widetilde{F} \subset \mathbb{C}^{g} \times \mathbb{C}^{g}$ whose intersection with $\mathcal{F} \times \mathcal{F}$ projects into $F$. From the definition in (\ref{R0constr}), this affine linear locus $\widetilde{F}$ intersects each $L_{n, \mathbf{a}}$ lying over the curve $C$ in a locus of dimension at most $0$. Moreover, the intersection $W \cap \widetilde{F}$ is non-empty (since it contains $V$). Consider the inverse image $\widetilde{E}$ of $\widetilde{F}$ inside $\mathcal{W}$. Then the locus $\widetilde{E}$ is algebraic, dominant onto $C_{\mathbb{C}}$, and intersects each fibre of $\mathcal{W} \to C_{\mathbb{C}}$ above $C$ in a locus of zero dimension; in particular, some component $\widetilde{E}^{\circ} \subset \widetilde{E}$ which whose image in $\mathbb{A}^{2g}$ contains $V$ has dimension one. Then the one-dimensional locus $V$ is contained in the image of the one-dimensional algebraic locus $\widetilde{E}^{\circ}$, hence $V$ is algebraic.

By the Ax-Lindemann theorem (a special case of the Ax-Schanuel theorem of \cite{zbMATH03407823}) since $V \subset \widetilde{Y} \times \widetilde{X}$ is algebraic, the image $\overline{\pi(V)}^{\textrm{Zar}}$ is weakly special in $A \times A$ and contained in $Y \times X$. On the other hand since $\dim (Y \times X) < \dim A$ and $V$ has positive dimension, this contradicts the simplicity of $A$ since when $A$ is simple any such weakly special must have dimension $\geq A$ (cf. the proof of \autoref{hodgeab}). We conclude that our original assumption was false, and so for any $\ep > 0$ there are $O(T^{\ep})$ points in $\mathcal{P}_{0}(T)$. 
\end{proof}

Now using \autoref{P0injlem} and \autoref{Toplyinn} the statement of \autoref{fewlineswhichintlem} implies the following:
\begin{quote}
The number of isolated non-torsion points of intersection between $\mathcal{I}_{\ell}$ for $\ell \leq n$ and $\widetilde{Y} \times \widetilde{X}$ is $O(n^{\ep})$ for any $\ep > 0$.
\end{quote}
Indeed, each such point of intersection gives rise to a point $r_{0} \in \mathcal{P}_{0}(\textrm{poly}(n))$, unique by \autoref{P0injlem}, and then \autoref{fewlineswhichintlem} imply there are at most $O(\textrm{poly}(n)^{\ep})$ such points, so the result follows after adjusting $\ep$. Moreover if a point of intersection between $\mathcal{I}_{\ell}$ and $\widetilde{Y} \times \widetilde{X}$ is non-isolated (i.e., there is a positive-dimensional component) then necessarily $\Gamma_{[\ell]}$ intersects $Y \times X$ in a positive-dimensional locus and agrees with a fibre of one of the $\rho_{i}$; this happens only for finitely many $\ell$. It then follows that 
\begin{quote}
The number of $\ell \leq n$ such that $\mathcal{I}_{\ell} \cap \widetilde{Y} \times \widetilde{X}$ contains a non-torsion intersection point is $O(n^{\ep})$ for any $\ep > 0$. In particular, such $\ell \in \N$ have asymptotic density zero.
\end{quote}
Now recall from \S\ref{introoutsidesec} that we have a natural density $1 - \delta$ set of integers in $\ell \in \N$ such that $\Gamma_{[\ell]}$ does not intersect $Y \times X$ in torsion points. It then follows from the above reasoning that for a density $1 - \delta$ collection of integers $\ell \in \N$ the loci $\mathcal{I}_{\ell}$ do not intersect $\widetilde{Y} \times \widetilde{X}$ at all; i.e., we have $Y \times X \cap \Gamma_{[\ell]} = \varnothing$ for a density $1 - \delta$ set of integers.

This concludes the proof of \Cref{densityofnonintthm}. \hfill \qedsymbol{} 

\begin{rem}
The statement of \Cref{densityofnonintthm} can be linked to the Mordell-Lang conjecture. Let $A$ be a simple abelian variety, $X$ an irreducible Zariski closed subvariety strictly contained in $A$, and $Y=\{P\}$ a complex non-torsion point of $A$. Then for all but finitely many $\lambda \in  \operatorname{End} (A)$, $X\cap \lambda Y=\emptyset$. To prove this, consider the finitely generated subgroup of $A(\C)$ given by 
\begin{displaymath}
\Gamma_P:=\{\lambda P, \lambda \in \operatorname{End} (A) \}
\end{displaymath}
The statement is equivalent to the claim that $X\cap \Gamma_P$ is finite. McQuillan and Faltings \cite{zbMATH00746072} \cite{zbMATH00721811} proved that the Zariski closure of $X\cap \Gamma_P$ is a finite union of weakly-special subvarieties, but since $A$ is simple, the only possibility is that the intersection is finite, as desired.

\end{rem}

\newpage

\bibliography{hodge_theory}
\bibliographystyle{abbrv}

\Addresses

\end{document}